\documentclass[11pt,reqno]{amsart}

\usepackage{amsmath,amsthm,amssymb,enumerate}
\usepackage{soul,color}
\usepackage[tmargin=1.2in,bmargin=1.2in,rmargin=1.2in,lmargin=1.2in]{geometry}
\usepackage[breaklinks=true]{hyperref}

\theoremstyle{plain}
\newtheorem{theorem}{Theorem}[section]

\newtheorem{proposition}[theorem]{Proposition}
\newtheorem{lemma}[theorem]{Lemma}

\theoremstyle{definition}
\newtheorem{definition}[theorem]{Definition}
\newtheorem{remark}[theorem]{Remark}
\newtheorem{example}[theorem]{Example}

\numberwithin{equation}{section}

\makeatletter
\let\c@theorem\c@table
\makeatother

\newcommand{\bb}{\mathbf{b}}
\newcommand{\br}{\mathbf{r}}
\newcommand{\bs}{\mathbf{s}}
\newcommand{\bu}{\mathbf{u}}
\newcommand{\bv}{\mathbf{v}}
\newcommand{\bw}{\mathbf{w}}
\newcommand{\bZ}{\mathbf{0}}
\newcommand{\cG}{\mathcal{G}}
\newcommand{\cK}{\mathcal{K}}
\newcommand{\cP}{\mathcal{P}}
\newcommand{\col}{\mathop{\mathrm{col}}}
\newcommand{\diag}{\mathop{\mathrm{diag}}\nolimits}
\newcommand{\Id}{\mathrm{Id}}
\newcommand{\onesymb}{\mathbf{1}}
\newcommand{\one}[1]{\onesymb_{#1}}
\newcommand{\stratumsymb}{\mathcal{S}}

\newcommand{\C}{\mathbb{C}}
\newcommand{\F}{\mathbb{F}}
\newcommand{\R}{\mathbb{R}}
\newcommand{\Z}{\mathbb{Z}}

\newcommand{\tu}[1]{\textup{#1}}

\begin{document}

\title{Simultaneous kernels of matrix Hadamard powers}

\author{Alexander Belton}
\address[A.~Belton]{Department of Mathematics and Statistics, Lancaster
University, Lancaster, UK}
\email{\tt a.belton@lancaster.ac.uk}

\author{Dominique Guillot}
\address[D.~Guillot]{University of Delaware, Newark, DE, USA}
\email{\tt dguillot@udel.edu}

\author{Apoorva Khare}
\address[A.~Khare]{Indian Institute of Science;
Analysis and Probability Research Group; Bangalore, India}
\email{\tt khare@iisc.ac.in}

\author{Mihai Putinar}
\address[M.~Putinar]{University of California at Santa Barbara, CA,
USA and Newcastle University, Newcastle upon Tyne, UK} 
\email{\tt mputinar@math.ucsb.edu, mihai.putinar@ncl.ac.uk}

\date{\today}



\begin{abstract}
In previous work [\textit{Adv.~Math.}~298:325--368, 2016], the structure
of the simultaneous kernels of Hadamard powers of any positive
semidefinite matrix were described. Key ingredients in the proof included
a novel stratification of the cone of positive semidefinite matrices and
a well-known theorem of Hershkowitz, Neumann, and Schneider, which
classifies the Hermitian positive semidefinite matrices whose entries are
$0$ or $1$ in modulus. In this paper, we show that each of these results
extends to a larger class of matrices which we term $3$-PMP (principal
minor positive).
\end{abstract}

\maketitle

\section{Introduction}

Given a positive integer $N$ and a subset $I \subset \C$,
let $\cP_N(I)$ denote the collection of $N \times N$ Hermitian
positive semidefinite matrices with all entries in $I$. Motivated by
the study of entrywise transformations of a matrix which preserve
positivity, the authors computed in \cite{BGKP-fixeddim} the
simultaneous kernel $\cK( A )$ of Hadamard powers
of a matrix $A = ( a_{ij} ) \in \cP_N(\C)$: that is,
\begin{equation}\label{Esimul}
\cK( A ) := \bigcap_{n \geq 0} \ker A^{\circ n}, 
\end{equation}
where $A^{\circ n} := ( a_{ij}^n )$. Here, we use the convention that
$A^{\circ 0} := \one{N}$, the rank-one $N \times N$ matrix with all
entries equal to $1$. Note that when the entrywise powers are replaced
by the conventional matrix powers and when $A$ is Hermitian, the kernels
$\ker A^n$ are all equal to $\ker A$ for $n \geq 1$. In contrast, the
simultaneous kernels of Hadamard powers are highly structured.

\begin{theorem}[{see \cite[Theorem~5.7]{BGKP-fixeddim}}]\label{Tsimul}
Let $A \in \cP_N( \C ) \setminus \{0\}$ and let $c_0, \dots, c_{N-1} \in
(0, \infty)$. Then 
\[
\cK( A ) = \ker \left(\sum_{j = 0}^{N - 1} c_j A^{\circ j}\right)
= \bigcap_{n=0}^{N-1} \ker A^{\circ n} = 
\bigoplus_{j = 1}^m \ker \one{I_j}
\] 
where $\pi^{\{1\}}(A) := \{I_1, \dots, I_m\}$ is a partition of $\{1,
\dots, N\}$ whose construction is described in Section~\ref{Sstrat}.
\end{theorem} 

The matrix $\sum_{j = 0}^{N - 1} c_j A^{\circ j}$ is obtained by
applying the polynomial $p(x) = \sum_{j=0}^{N-1} c_j x^j$ to the
entries of $A$. The first equality thus provides a connection between
the study of simultaneous kernels and the study of polynomials that
preserve positivity when applied entrywise to matrices in $\cP_N(\C)$;
see \cite{BGKP-fixeddim, BGKP-fpsac} for more details.

An intriguing consequence of Theorem~\ref{Tsimul} is the rigidity of the
simultaneous kernels: there are only finitely many possibilities for
$\cK(A)$. This is in stark contrast to the case where Hadamard powers are
replaced by conventional powers, in which case any subspace of $\C^N$ can
obviously arise as a simultaneous kernel. 

The proof of Theorem~\ref{Tsimul} in \cite{BGKP-fixeddim}
was lengthy and involved, with two main ingredients.
\begin{enumerate}
\item[(a)] An observation of Hershkowitz, Neumann, and Schneider
\cite{Matrix01psd}, which classifies Hermitian positive semidefinite
matrices whose entries are $0$ or $1$ in modulus.
\item[(b)] A novel Schubert cell-type stratification for Hermitian
positive semidefinite matrices \cite[Theorem~5.1]{BGKP-fixeddim}. 
\end{enumerate}

We show in the present note how these two ingredients, as well as
Theorem~\ref{Tsimul}, can be extended to a much broader class of
Hermitian matrices.  Namely, we show that these results hold for any
matrix whose principal minors of size at most $3$ are
non-negative. Matrices satisfying the latter property will be termed
$3$-PMP (\emph{principal minor positive}). Apart from proving
Theorem~\ref{Tsimul} in greater generality, our new approach
simplifies the original proof \cite{BGKP-fixeddim}.

The article is structured as follows. In Section~\ref{Shns}, we recall
the main result of Hershkowitz--Neumann--Schneider \cite{Matrix01psd}
and extend it to $3$-PMP matrices. A key component of their original
proof is the \emph{principal submatrix rank property} (PSRP). We
therefore examine the relationship between the PSRP and the notion
of principal minor positivity. In Section~\ref{Ssign}, we identify
precisely how principal minor positivity constrains the signature of a
Hermitian matrix. In Section~\ref{Sstrat}, we recall and extend to all
Hermitian matrices the Schubert cell-type stratification of the cone
of positive semidefinite matrices that was developed in
\cite{BGKP-fixeddim}. Section~\ref{Ssimul} concludes the paper by
classifying the simultaneous kernels of Hadamard powers of all $3$-PMP
matrices.

\section{The Hershkowitz--Neumann--Schneider theorem and the principal
submatrix rank property}\label{Shns}

We start by isolating the main class of matrices of interest in this
paper.

\begin{definition}
Let $1 \leq k \leq N$. A Hermitian matrix $A \in \C^{N \times N}$ will be
termed \emph{$k$-PMP} if every principal $j \times j$ minor of $A$ is
non-negative, for $j = 1$, \dots, $k$.
\end{definition}

The notion of principal minor positivity interpolates
between Hermitian matrices, which are $0$-PMP by convention, and
positive semidefinite matrices, which are $N$-PMP.

\begin{remark}[Examples and special cases]\label{Rexample}
Given $N \geq 1$, let $A := \lambda \Id_N - \one{N}$,
where~$\Id_N$ is the $N \times N$ identity matrix and $\one{N}$ is
the rank-one $N \times N$ matrix with all entries equal to~$1$. It is
readily seen that if $k \in \{ 1, \dots, N \}$ and
$\lambda \in [ k - 1, k )$, then the matrix $A$ is $( k - 1 )$-PMP but
not $k$-PMP. Thus the successive inclusions in the PMP-filtration of
all Hermitian matrices are strict. R.~B.~Bapat pointed us to work of
Mohan, Parthasarathy, and Sridhar, who introduced the notion of
\emph{$P$-matrices of exact order $N-k$}; see~\cite{MPS}. These
are a refinement of $k$-PMP matrices, whose principal minors of size
no more than $k$ are positive, and those of size greater than $k$ are
negative. The matrix $A := \lambda \Id_N - \one{N}$ is a
$P$-matrix of exact order $N - k$ whenever $\lambda \in ( k - 1, k )$.
\end{remark}

Our goal in this paper is to compute the simultaneous kernel for the
entrywise powers of any given $3$-PMP matrix; these
comprise a much larger family of matrices than the cone $\cP_N( \C )$
that was considered in \cite{BGKP-fixeddim}. As noted in the
Introduction, the following Theorem is an important first step.

\begin{definition}[{Hershkowitz--Neumann--Schneider
\cite[Definition~2.1]{Matrix01psd}}]\label{Dperm}
A matrix $P \in \C^{N \times N}$ is called a \emph{unitary monomial
matrix} if $P = Q D$, where $Q$ is a permutation matrix and $D$ is a
diagonal matrix all of whose diagonal entries are of modulus $1$.
\end{definition}

\begin{theorem}[{Hershkowitz--Neumann--Schneider
\cite[Theorem~2.2]{Matrix01psd}}]\label{Thns}
A matrix $A \in \C^{N \times N}$ is Hermitian positive semidefinite and
all its entries have modulus $1$ or $0$ if and only if $A$ is similar, by
means of a unitary monomial matrix, to a direct sum of matrices each of
which is either a matrix with entries all equal to~$1$ or a zero matrix.
\end{theorem}

Positive semidefinite matrices with entries in $S^1 \sqcup \{ 0 \}$,
as well as $3$-PMP matrices, occur in many settings. For example, they
naturally feature in Horn's study of the incidence matrices of
infinitely divisible matrices and kernels \cite[Theorem~1.13]{horn}. A
simple application of Theorem~\ref{Thns} also shows that graphs with
smallest adjacency eigenvalue $-1$ are disjoint unions of complete
graphs \cite{brouwer2011spectra}. Properties of matrix thresholding
were also derived in \cite{Guillot_Rajaratnam2012} using related
ideas.

\begin{remark}
It is worth mentioning that the unitary monomial matrices of
Definition~\ref{Dperm} were discovered by Banach in connection with a
typical early functional-analysis question. More precisely, these
matrices appear in the structure of the group $\cG_p$ of linear
isometries of~$\R^N$ or~$\C^N$ equipped with the $\ell^p$~norm, for
$p \in [1,\infty]$. (Thanks to the Mazur--Ulam theorem, in the real
case every isometry belongs to this group, up to a translation; the
existence of the canonical conjugation shows this is false in the
complex case.)

As is well known, the group $\cG_2$ is the orthogonal or unitary
group, respectively. If \mbox{$p \neq 2$}, then results of Banach and
Lamperti \cite{Banach,Lamperti} imply that the group $\cG_p$ does not
depend on~$p$, and consists of precisely the \emph{generalized
permutation matrices}, i.e., the products of permutation matrices with
diagonal orthogonal or unitary matrices.  Specifically, for $\R^N$ we
obtain the hyperoctahedral group of $2^N N!$ signed permutations,
i.e., the Weyl group of type $B_N = C_N$. In the complex case, the
group $\cG_p$ is composed of all matrices which are products of
permutations and diagonal unitary matrices, i.e., the unitary monomial
matrices. (Elementary proofs of these special cases of the
Banach--Lamperti theorem have been found.)
\end{remark}

Returning to Theorem~\ref{Thns}, the proof by Hershkowitz, Neumann,
and Schneider is rather intriguing, relying on the following notion.

\begin{definition}
A matrix $M \in \C^{N \times N}$ is said to satisfy the \emph{principal
submatrix rank property} (\emph{PSRP}) if the following conditions hold.
\begin{enumerate}
\item The column space determined by every set of rows of $M$ is equal to
the column space of the principal submatrix lying in these rows.
\item The row space determined by every set of columns of $M$ is equal to
the row space of the principal submatrix lying in these columns.
\end{enumerate}
\end{definition}

Hermitian positive semidefinite matrices satisfy the PSRP. For the
convenience of the reader, we reproduce the argument provided in
\cite{Matrix01psd} to prove this claim. Let $M \in \cP_N(\C)$ and
write $M$ as
\[
M = \begin{pmatrix}
A & B \\
B^* & C
\end{pmatrix},
\]
where $A \in \cP_k(\C)$ and $1 \leq k < N$. It is enough to show
that the row space of $\begin{pmatrix} A \\ B^* \end{pmatrix}$ is equal to
the row space of $A$. Working with orthogonal complements, this is
equivalent to the following: If
$\bv = \begin{pmatrix} \bw\\ \bZ \end{pmatrix} \in \C^N$ and
$A \bw = \bZ$, then $M \bv = \bZ$. But for such a vector $\bv$ we have
that~$\bv^* M \bv = \bw^*A \bw = 0$, and the result follows by the
positive semidefiniteness of $M$.

We include below a short sketch of the proof from \cite{Matrix01psd} to
illustrate how the principal submatrix rank property is used in the proof
of Theorem~\ref{Thns}. The reader is referred to \cite{Matrix01psd} for
the details.

\begin{proof}[Sketch of the proof of Theorem~\ref{Thns}]
The proof proceeds by induction, with the case $n = 1$ trivial. Since the
leading $(n-1) \times (n-1)$ principal submatrix $B$ of $A$ is positive
semidefinite, there is a unitary monomial matrix $P$ such that
$P^{-1} B P$ is a direct sum of~$\one{k}$ matrices and a zero matrix.
If $R = P \oplus 1$, then $C = R^{-1} A R$ is such that all non-zero
elements in the last row and column of $C$ are of modulus $1$ and the
leading $( n - 1 ) \times ( n - 1 )$ submatrix of~$C$ is a direct sum
of the desired form. We now partition the last row and column of $C$
in conformity with this direct sum. Since $C$ is positive
semidefinite, it follows by the PSRP that each subvector of the last
row and column of $C$ determined by this partition is a multiple of
the vector with entries all~$1$ of the appropriate size by a number of
modulus $1$ or by $0$. The conclusion follows.
\end{proof}

Next we show how Theorem~\ref{Thns} can naturally be extended to $3$-PMP
matrices.

\begin{theorem}\label{TnewHNS}
Given a Hermitian matrix $A \in \C^{N \times N}$, the following are
equivalent.
\begin{enumerate}
\item[\tu{(1)}] The matrix $A$ is $3$-PMP with entries of modulus $0$ or
$1$.
\item[\tu{(2)}] There exist a diagonal matrix $D$, whose diagonal
entries lie in $S^1$, and a permutation matrix $Q$, such that
$( Q D )^{-1} A ( Q D )$ is a block-diagonal matrix with each diagonal
block a square matrix of either all ones or all zeros.
\item[\tu{(3)}] The matrix $A \in \cP_N(\C)$ with entries of modulus $0$
or $1$.
\end{enumerate}
\end{theorem}

Note that Theorem~\ref{TnewHNS} fails for the Hermitian matrix
\[
A = %
\begin{pmatrix} \phantom{-}1 & \phantom{-}1 & -1 \\
 \phantom{-}1 & \phantom{-}1 & \phantom{-}1 \\
 -1 & \phantom{-}1 & \phantom{-}1 \end{pmatrix},
\]
which is $2$-PMP but not $3$-PMP. Thus Theorem~\ref{TnewHNS} has no
immediate generalization.

\begin{proof}[Proof of Theorem~\ref{TnewHNS}]
If (2) holds then $( Q D )^{-1} A ( Q D )$ is positive semidefinite and
so $A$ is also positive semidefinite. Thus (2) implies (3). Clearly (3)
implies (1).

Finally, suppose (1) holds. Then $a_{j j} \in \{ 0, 1 \}$ for
all~$j$, and the cases $N = 1$ and $N = 2$ are readily verified;
suppose $N \geq 3$. Note that $A$ has no principal $3 \times 3$
submatrix of the form
\[
C = \begin{pmatrix} 1 & a & 0 \\
 \overline{a} & 1  & b \\
 0 & \overline{b} & 1 \end{pmatrix}, \qquad %
\text{where } a, b \in S^1,
\]
since $\det C = 1 - b \overline{b} - a \overline{a} = -1$. Thus the
non-zero entries of $A$ can be permuted into a block-diagonal matrix
by conjugation with a permutation matrix $Q$. It remains to show each
non-zero diagonal block has rank one, since if these blocks are of the
form~$\bu_j \bu_j^*$, with the entries of~$\bu_j$ in~$S^1$, then we
can write the concatenation of the~$\bu_j$ as the diagonal entries of
a diagonal matrix~$D$, and the result follows.

We suppose henceforth that every entry of $A$ has modulus one, and
claim that $A = \bu \bu^*$, where
$\bu = ( a_{1 1}, \dots, a_{1 N} )^*$, i.e., that
$a_{i j} = \overline{a_{1 i}} a_{1 j}$ for all $i$ and $j$. As~$A$ is
$3$-PMP, the principal minor
\[
\begin{vmatrix} 1 & a_{1 i} & a_{1 j} \\
 a_{i 1} & 1  & a_{i j} \\
 a_{j 1} & a_{j i} & 1 \end{vmatrix} = %
2 \Re( a_{1 i} a_{i j} a_{j 1} ) - 2
\]
is non-negative, so $a_{1 i} a_{i j} a_{j 1} = 1$ and $a_{i j} =
\overline{a_{1 i} a_{j 1}} = \overline{a_{1 i}} a_{1 j}$. This proves
$(1) \implies (2)$, and concludes the proof.
\end{proof}

Theorem~\ref{TnewHNS} uses the $3$-PMP property instead of the
PSRP. It is therefore natural to ask how different these two
properties are. To explore this question, we introduce a refinement of
the principal submatrix rank property.
 
\begin{definition}
A matrix $M \in \C^{N \times N}$ is said to satisfy the
\emph{$k$-principal submatrix rank property} (\emph{$k$-PSRP}) if
the following conditions hold.
\begin{enumerate}
\item The column space determined by every set of $k$ rows of $M$ is
equal to the column space of the principal submatrix lying in these
rows.
\item The row space determined by every set of $k$ columns of $M$ is
equal to the row space of the principal submatrix lying in these
columns.
\end{enumerate}
\end{definition}

Our next result shows that a $k$-PMP matrix satisfies the $(k-1)$-PSRP.

\begin{theorem}\label{Tpsrp}
Let $M \in \C^{N \times N}$ be a $k$-PMP Hermitian matrix, where
$2 \leq k \leq N$. Then $M$ satisfies the $l$-PSRP for all $l < k$. 
\end{theorem}

\begin{proof}
Since a $k$-PMP matrix is also $l$-PMP for all $l < k$, it suffices to
take $l = k-1$. Without loss of generality, we show that
the column space determined by the first $k-1$ rows of $M$ is equal to
the column space of the leading principal $(k-1) \times (k-1)$ submatrix
of $M$; the general case follows by simultaneously
permuting the rows and columns of $M$. Write 
\[
M = \begin{pmatrix}
A & B \\
B^* & C
\end{pmatrix},
\]
with $A \in \C^{(k-1) \times (k-1)}$,
$B \in \C^{(k-1) \times (N-k+1)}$, and
$C \in \C^{(N-k+1) \times (N-k+1)}$. Denote the
columns of $B$ by $\bb_k$, $\bb_{k+1}$, \dots, $\bb_N \in \C^{k-1}$,
and let $\col(A)$ denote the column space of~$A$. We
will prove that $\bb_i \in \col(A)$ whenever $k \leq i \leq N$. Indeed,
for such $i$, let
\[
M_i := \begin{pmatrix}
A & \bb_i \\
\bb_i^* & m_{ii}
\end{pmatrix} \in \C^{k \times k}
\]
be the principal submatrix of $M$ formed by its rows and columns
numbered $1$, \dots, $k-1$, and~$i$. By assumption,
$M_i \in \cP_k(\C)$.
If $m_{ii} = 0$, then
$\bb_i = \bZ \in \col(A)$, which may be seen by
inspecting $2 \times 2$ principal minors of $M_i$. Otherwise, the
Schur complement of $m_{ii}$ in~$M_i$ is positive semidefinite, i.e.,
\begin{equation}\label{Eschur}
A - \frac{1}{m_{ii}} \bb_i \bb_i^* \in \cP_{k-1}(\C).
\end{equation}
If $\bb_i \not\in \col(A)$, there exists
$\bv \in \col(A)^\perp = \ker A$ such that $\bv^* \bb_i \neq 0$. For
such a vector $\bv$, we have that
\[
\bv^* \Bigl( A - \frac{1}{m_{ii}} \bb_i \bb_i^* \Bigr) \bv = 
-\frac{1}{m_{ii}} | \bv^* \bb_i |^2 < 0,
\]
contradicting Equation (\ref{Eschur}). We conclude that
$\bb_i \in \col(B)$, as claimed.
\end{proof}

Theorem~\ref{Tpsrp} shows that the $k$-PMP property imposes 
constraints on the row and column spaces of the matrix. However, our
next result shows that there exist $k$-PMP matrices with an arbitrarily
large gap between the dimension of the column space determined by a set
of~$l$ rows, and the rank of the principal submatrix lying in these rows.
Thus there is a major discrepancy between the $3$-PMP property and the
PSRP.

\begin{theorem}\label{Tpsrp2}
Let $2 \leq k \leq l < N$. Then there exists a real
symmetric matrix 
\begin{equation}\label{Eblock}
M = \begin{pmatrix}
A & B \\
B^T & C
\end{pmatrix} \in \R^{N \times N},
\end{equation}
with $A \in \R^{l \times l}$, such that
\begin{enumerate}
\item[\tu{(1)}] the matrix $M$ is $k$-PMP,
\item[\tu{(2)}] the matrix $A$ has rank $k - 1$, and
\item[\tu{(3)}] the matrix
$\begin{pmatrix} A & B\end{pmatrix}$ has rank
$\min\{ l, k - 1 + N - l \}$.
\end{enumerate}
In particular, the matrix $M$ does not satisfy the $l$-PSRP.
\end{theorem}

The following simple lemma is crucial to our proof of
Theorem~\ref{Tpsrp2}.

\begin{lemma}\label{Lpsrp}
Let $1 \leq m \leq l$. There exists a positive semidefinite matrix
$A \in \cP_l(\R)$ with rank $m$, such that the $p \times p$ principal
minors of $A$ are strictly positive whenever $1 \leq p \leq m$.
\end{lemma}

\begin{proof}
Let $\bu = ( u_1, \ldots, u_l )^T \in \R^l$ be a vector
with distinct non-zero entries. Define
\[
A = \sum_{i=0}^{m-1} \bu^{\circ i} ( \bu^{\circ i} )^T, 
\]
where $\bu^{\circ i}$ is the vector with components
$( u_j^i )_{1 \leq j \leq l}$. Then $A$ is positive semidefinite, and
the desired properties follow immediately from the non-singularity of
Vandermonde matrices.
\end{proof}

\begin{proof}[Proof of Theorem~\ref{Tpsrp2}]
Let $A \in \cP_l(\R)$ be the matrix with rank~$m = k-1$
provided by Lemma~\ref{Lpsrp}. Choose vectors $\bu_{l+1}$, \dots,
$\bu_N \in \ker A$ such that their span has the largest dimension
possible, i.e., $\min\{ l - k + 1, N - l \}$. For $\epsilon_{l+1}$,
\dots, $\epsilon_N > 0$, let $B$ be the matrix with columns
$\epsilon_{l+1} \bu_{l+1}$, \dots, $\epsilon_{N} \bu_N$. Also,
let $C = \Id_{N - l}$, the $N - l$-dimensional identity matrix. Then
the matrix $M$ given by Equation~(\ref{Eblock}) satisfies properties
(2) and (3) of the theorem.

We claim that $M$ is also $k$-PMP if $\epsilon_{l+1}$, \dots,
$\epsilon_N$ are small enough. To prove this claim,
let~$I := \{ i_1, \dots, i_k \}$
be a subset of $\{ 1, \dots, N \}$ of cardinality $k$, and let $M_I$
denote the principal submatrix of $M$ formed by restricting $M$ to the
rows and columns in $I$. Clearly, $M_I$ is positive semidefinite if
either $I \subset \{ 1, \dots, l \}$ or $I \subset \{ l+1, \dots, N \}$.
Now assume $I$ contains elements of both sets, so
that~$I = I_1 \sqcup I_2$, with $I_1 \subset \{ 1, \dots, l \}$,
$I_2 \subset \{ l + 1, \dots, N \}$, and both $I_1$ and $I_2$ are
non-empty. Note that $M_{I_2} = \Id_{|I_2|}$. Thus $M_I$ is positive
semidefinite as long as the Schur complement
$M_{I_1} - M_{I_1, I_2} M_{I_1, I_2}^T$ is positive
semidefinite, where $M_{I_1, I_2}$ denotes the submatrix of $M$ with rows
in $I_1$ and columns in $I_2$. Note that $M_{I_1} = A_{I_1}$ is positive
definite by Lemma~\ref{Lpsrp}. By a continuity argument,
it follows that~$M_{I_1} - M_{I_1, I_2} M_{I_1, I_2}^T$ is also positive
definite if $0 < \epsilon_{l+1}, \dots, \epsilon_N < C_{I_1}$, for some
positive constant $C_{I_1}$ that depends only on $I_1$.
Let $C := \min_{J : |J| < k} C_J$, where the minimum is taken over all
subsets $J \subset \{ 1, \dots, l \}$ of size less than $k$. Then
$C > 0$ since there are finitely many such subsets. If
$0 < \epsilon_{l+1}$, \dots, $\epsilon_N < C$, this argument shows
that $M_I$ is positive semidefinite for any
set~$I \subset \{ 1, \dots, N \}$ of cardinality~$k$. We conclude that
$M$ is $k$-PMP, as required.
\end{proof}

\section{Principal minor positivity and signature}\label{Ssign}

We next explore how the property of a matrix being $k$-PMP affects its
signature, and so its rank. In the most restrictive case, if an
$N \times N$ matrix is $N$-PMP, then it has no negative
eigenvalues; moreover, every admissible signature
$(n_+, n_0 = N - n_+, n_- = 0)$ is attained, which may be seen by
considering the block-diagonal matrix
$\Id_{n_+} \oplus \bZ_{n_0}$.

The purpose of this section is to consider all other Hermitian matrices.
We now prove:

\begin{theorem}\label{Tpmp-signature}
Fix integers $k$ and $N$, with $0 \leq k < N$.
\begin{enumerate}
\item[\tu{(1)}] Let $A \in \C^{N \times N}$ be $k$-PMP but not
$(k+1)$-PMP. Then $A$ has at least $k$ positive eigenvalues, and at
least one negative eigenvalue.

\item[\tu{(2)}] Conversely, for every signature $( n_+, n_0, n_- )$
such that 
\[
n_+ \geq k, \quad n_0 = N - n_+ - n_- \geq 0, %
\quad \text{and} \quad n_- \geq 1,
\]
there exists a matrix $A \in \C^{N \times N}$ with this signature, and
such that $A$ is $k$-PMP but not $(k+1)$-PMP.
\end{enumerate}
\end{theorem}

\begin{proof}
We begin by showing (1). By assumption, there exists a
$(k+1) \times (k+1)$ principal minor of $A$, say $B$, such that
$\det B < 0$. As $B$ is $k$-PMP, it follows by Cauchy interlacing
that~$B$ has one negative eigenvalue and $k$ positive ones. Another
application of Cauchy interlacing now gives the same statement for $A$.

We next prove the converse result (2). By Lemma \ref{Lpsrp}, there exists
a matrix $B \in \cP_{k+n_-}(\R)$ of rank $k$ such that all $p \times p$
principal minors of $B$ are strictly positive, for
$1 \leq p \leq k$. Now let $P \in \cP_{k+n_-}(\R)$ denote the
projection matrix of the subspace $\ker B$, and define
\[
B_\epsilon := B - \epsilon P, \qquad \text{where } \epsilon > 0.
\]
By the continuity of determinants, we may fix $\epsilon$ sufficiently
small to ensure that all $p \times p$ principal minors of $B_\epsilon$
are positive, for $1 \leq p \leq k$, so that $B_\epsilon$ is
$k$-PMP. Furthermore, it is clear by diagonalization that $B_\epsilon$
has signature $( k, 0, n_- )$.

We now show that $B_\epsilon$ is not $(k+1)$-PMP. In fact, we claim 
that every $(k+1) \times (k+1)$ principal minor of $B_\epsilon$ is
negative. If $\det M \geq 0$ for some $(k+1) \times (k+1)$ principal
minor $M$ of $B_\epsilon$, then $M$ is positive semidefinite and so has
$k+1$ non-negative eigenvalues. This last statement then holds for
$B_\epsilon$, by Cauchy interlacing, which contradicts the fact that
$B_\epsilon$ has signature $( k, 0, n_- )$.

The result now follows by taking
$A := B_\epsilon \oplus \Id_{n_+ - k} \oplus \bZ_{N - n_+ - n_-}$.
\end{proof}

\section{Schubert cell-type stratification of $3$-PMP
matrices}\label{Sstrat}

In this section we explain and generalize the novel Schubert cell-type
stratification of~$\cP_N(\C)$ uncovered in \cite{BGKP-fixeddim},
which plays a crucial role in determining the simultaneous kernel
$\cK(A)$ defined in Equation~(\ref{Esimul}).

\begin{theorem}[{\cite[Theorem~5.1]{BGKP-fixeddim}}]%
\label{Tgroup}
Fix a multiplicative subgroup $G \subset \C^\times$, an integer
$N \geq 1$, and a non-zero matrix $A \in \cP_N( \C )$.
\begin{enumerate}
\item[\tu{(1)}] Suppose $\{ I_1, \dots, I_m \}$ is a partition of
$\{ 1, \dots, N \}$ satisfying the following conditions.
\begin{enumerate}
\item[\tu{(a)}] Each diagonal block $A_{I_j}$ of $A$ has rank at most
one, and $A_{I_j} = \bu_j \bu_j^*$ for a unique vector
$\bu_j \in \C^{| I_j |}$ with first entry
$\bu_{j, 1} \in [ 0, \infty )$.

\item[\tu{(b)}] The entries of each diagonal block $A_{I_j}$ lie in a
single $G$-orbit.
\end{enumerate}
Then there exists a unique matrix $C = ( c_{i j})_{i,j = 1}^m$ such
that $c_{i j} = 0$ unless $\bu_i \neq 0$ and $\bu_j \neq 0$, and $A$
is a block matrix with
\[
A_{I_i \times I_j} = c_{i j} \bu_i \bu_j^* \qquad %
( 1 \leq i, j \leq m ).
\]
Moreover, the entries of each off-diagonal block of $A$ also
lie in a single $G$-orbit. The matrix
$C \in \cP_m( \overline{D}(0,1))$, and the matrices $A$ and $C$ have
equal rank.

\item[\tu{(2)}] Consider the condition \tu{(c)}.
\begin{enumerate}
\item[\tu{(c)}] The diagonal blocks of $A$ have maximal size, i.e.,
each diagonal block is not contained in a larger diagonal block that
has rank one.
\end{enumerate}
There exists a partition $\{ I_1, \dots, I_m \}$ such that \tu{(a)},
\tu{(b)} and \tu{(c)} hold, and such a partition is unique up to
relabelling of the indices.

\item[\tu{(3)}] Suppose \tu{(a)}, \tu{(b)}, and \tu{(c)} hold, and
$G = \C^\times$. Then the off-diagonal entries of $C$ lie in the open
disc $D( 0, 1 )$.

\item[\tu{(4)}] If $G \subset S^1$, then diagonal blocks of $A$ in a
single $G$-orbit have rank at most one.
\end{enumerate}
\end{theorem}

Note that Theorem~\ref{Tgroup}(4) is an immediate consequence of
Theorem~\ref{Thns}.

\begin{example}
To illustrate Theorem~\ref{Tgroup}, consider the following
$5 \times 5$ Hermitian matrix,
\[
A = \begin{pmatrix}
2 & 2 & 1 & -2i & 2 \\
2 & 2 & 1 & -2i & 2 \\
1 & 1 & 1 & -i & 1 \\
2i & 2i & i & \hphantom{-}2 & 2i \\
2 & 2 & 1 & -2i & 2
\end{pmatrix}.
\]
It is readily verified that $A \in \cP_5(\C)$. Let $G_1 = \{1\}$, the
trivial multiplicative subgroup of~$\C^\times$, and consider the
partition $\pi_1 := \{ \{ 1, 2, 5 \}, \{ 3 \}, \{ 4 \} \}$.
Permuting the rows and columns of $A$ according to $\pi_1$, we obtain
\[
A' := \begin{pmatrix}
2 & 2 & 2 & 1 & -2i \\
2 & 2 & 2 & 1 & -2i \\
2 & 2 & 2 & 1 & -2i \\
1 & 1 & 1 & 1 & -i \\
2i & 2i & 2i & i & \hphantom{-}2
\end{pmatrix}.
\]
It follows immediately that $\pi_1$ is the unique partition, up to
relabelling of the indices, afforded by
Theorem~\ref{Tgroup}. Similarly, consider the cyclic subgroup
$G_2 = \{ 1, -1, i, -i \} \subset \C^\times$, and let the partition
$\pi_2 := \{ \{ 1, 2, 4, 5 \}, \{ 3 \} \}$. Permuting the rows and
columns of $A$ with respect to~$\pi_2$, we obtain
\[
A'' := \begin{pmatrix}
2 & 2 & -2i & 2 & 1 \\
2 & 2 & -2i & 2 & 1 \\
2i & 2i & \hphantom{-}2 & 2i & i \\
2 & 2 & -2i & 2 & 1 \\
1 & 1 & -i & 1 & 1
\end{pmatrix}.
\]
The $4 \times 4$ leading principal submatrix of $A''$ has rank $1$,
and its entries lie in the same $G_2$-orbit. It follows that
$\pi_2$ is the partition provided by Theorem~\ref{Tgroup} in this
case. Finally, if we take $G_3 = S^1 \subset \C^\times$, then we
recover the same block structure and partition as for $G_2$.

For all of these subgroups, the entries of each off-diagonal block
belong to a single orbit, as guaranteed by the theorem.
\end{example}

Theorem~\ref{Tgroup} provides a natural stratification of the cone
$\cP_N( \C )$.

\begin{definition}\label{Dpsd-stratum}
Let $\Pi_N$ denote the set of all partitions of $\{ 1, \dots, N \}$,
partially ordered so that~$\pi' \prec \pi$ if and only if $\pi$ is a
refinement of $\pi'$. Given a matrix $A \in \cP_N( \C )$ and
$G \subset S^1$, define $\pi^G( A ) \in \Pi_N$ to be the partition
provided by Theorem~\ref{Tgroup} for the matrix $A$. Conversely, for a
partition $\pi \in \Pi_N$ and $G \subset S^1$, let the \emph{stratum}
\[
\stratumsymb^G_\pi := \{ A \in \cP_N( \C ) : \pi^G( A ) = \pi \}.
\]
\end{definition}

The Schubert property is reflected in the decomposition
$\overline{\stratumsymb^G_\pi} = %
\bigsqcup_{\pi' \prec \pi} \stratumsymb^G_{\pi'}$,
where $\pi'$ runs over all coarsenings of $\pi$ in $\Pi_N$.

In \cite[Section~5]{BGKP-fixeddim}, it was shown that for
any matrix $A \in \cP_N( \C )$, the simultaneous kernel of its
entrywise powers,
\[
\cK( A ) := \bigcap_{n \geq 0} \ker A^{\circ n},
\]
is explicitly computable, and equals the kernel of a single matrix that
depends on $A$ only through the partition~$\pi^{\{ 1 \}}( A )$; see
Theorem~\ref{Tsimul}. Consequently, this simultaneous kernel
$\cK( A )$ is unchanged as~$A$ runs through a fixed
stratum~$\stratumsymb^{\{1\}}_\pi$ in~$\cP_N( \C )$.

In order to prove a version of Theorem~\ref{Tsimul} for $3$-PMP
matrices, we first extend Theorem~\ref{Tgroup}.

\begin{theorem}\label{Tpmpstrata}
Suppose $G \subset \C^\times$ is a multiplicative subgroup and
$3 \leq k \leq N$. The assertions in Theorem~\ref{Tgroup} all hold if
$A$ is only required to be $k$-PMP, except that the matrix $C$ is now
only assured to be $k$-PMP, rather than positive semidefinite.
\end{theorem}

\begin{proof}
All assertions but the last follow as in the proof of
Theorem~\ref{Tgroup} given in \cite{BGKP-fixeddim}, using the $3$-PMP
property and Theorem~\ref{TnewHNS} in place of Theorem~\ref{Thns}
for~(4). That $C$ inherits the $k$-PMP property from $A$ is readily
verified.
\end{proof}

We will refer to the analogous statement to Theorem~\ref{Tgroup}(1)
provided by Theorem~\ref{Tpmpstrata} as Theorem~\ref{Tpmpstrata}(1),
and similarly for the other parts of these theorems.

\begin{definition}
Suppose $G \subset \C^\times$ is a multiplicative subgroup and
$A \in \C^{N \times N}$ is $k$-PMP, where $3 \le k \le N$. The
partition given by Theorem~\ref{Tpmpstrata}(2) is
denoted~$\pi^G( A )$.
\end{definition}

Theorem~\ref{Tpmpstrata} immediately leads to stratifications of the
$k$-PMP matrices, for $3 \leq k \leq N$. These stratifications respect
the natural inclusions, so that $\pi^G( A )$ is independent of $k$.

This common Schubert cell-type stratification in fact holds more
generally, for all Hermitian matrices.

\begin{proposition}\label{Pexists}
Fix a multiplicative subgroup $G \subset \C^\times$, an
integer $N \geq 1$, and a Hermitian matrix $A \in \C^{N \times N}$. There
exists a coarsest partition
$\pi_{\min}( A ) = \{ I_1, \dots, I_m \} \in \Pi_N$, such that the
entries of the block submatrix $A_{I_i \times I_j}$ lie in
a single $G$-orbit, for all $i$, $j \in \{ 1, \dots, m \}$.
This partition is unique up to relabelling of the indices.

If, moreover, $A$ is $3$-PMP and $G \subset S^1$, then
$\pi_{\min}( A ) = \pi^G( A )$.
\end{proposition}

Note that $3$-PMP matrices differ from Hermitian matrices
which are not $3$-PMP, in that even for $G \subset S^1$, every block
given by the partition $\pi_{\min}$ can have rank greater than $1$. For
example, the following Hermitian matrix $A$ is $2$-PMP but not
$3$-PMP, and for $G = \{ 1,-1 \}$ we have that
$\pi_{\min}( A ) = \{ \{ 1, 2, 3 \}, \{ 4, 5, 6 \} \}$, with all four
blocks of $A$ non-singular:
\[
A = \begin{pmatrix}
\hphantom{-}2 & 2 & -2 & \hphantom{-}1 & 1 & -1\\
\hphantom{-}2 & 2 & \hphantom{-}2 &  \hphantom{-}1 & 1 & \hphantom{-}1\\
-2 & 2 &  \hphantom{-}2 & -1 & 1 &  \hphantom{-}1\\
\hphantom{-}1 & 1 & -1 &  \hphantom{-}2 & 2 & -2\\
\hphantom{-}1 & 1 & \hphantom{-}1 & \hphantom{-}2 & 2 & \hphantom{-}2\\
-1 & 1 &  \hphantom{-}1 & -2 & 2 & \hphantom{-}2
\end{pmatrix}.
\]

\begin{proof}[Proof of Proposition~\ref{Pexists}]
The non-trivial part is to establish uniqueness. To do so, we first
claim that if $\pi_1$, $\pi_2 \in \Pi_N$ satisfy the property in the
assertion, then so does their meet $\pi_1 \wedge \pi_2$; note this
gives uniqueness, since minimal $\pi_1$ and $\pi_2$ are
such that $\pi_1 = \pi_1 \wedge \pi_2 = \pi_2$.

To show the claim, first recall that the meet $\pi_1 \wedge \pi_2$ can
be constructed as follows:
connect vertices $i$, $i' \in \{ 1, \dots, N \}$ by an edge if they
lie in the same block of $\pi_1$ or $\pi_2$; this defines a graph
whose connected components are the blocks of the partition
$\pi_1 \wedge \pi_2$. Denote this equivalence relation
by $i \sim i'$ in~$\pi_1 \wedge \pi_2$.

Now suppose $i \sim i'$ and $j \sim j'$ in $\pi_1 \wedge \pi_2$, so
there are paths joining them, each of whose vertices lies in a block
of $\pi_1$ or $\pi_2$. Denote these paths by
\[
i = i_0 \leftrightarrow i_1 \leftrightarrow \dots \leftrightarrow i_r =
i' \quad \text{and} \quad
j = j_0 \leftrightarrow j_1 \leftrightarrow \dots \leftrightarrow j_s =
j'.
\]
We claim that $a_{i j} \in G \cdot a_{i' j'}$. Indeed, using the above
paths,
\[
a_{i j} = a_{i_0 j} \in G a_{i_1 j} = G a_{i_2 j} = \dots = %
G a_{i_r j} = G a_{i' j_0} = G a_{i' j_1} = \dots = G a_{i' j'}.
\]
This proves the claim, and hence all but the last assertion.

If $A$ is $3$-PMP and $G \subset S^1$, then Theorem~\ref{TnewHNS}
gives that $\pi_{\min}( A )$ satisfies properties (a) and (b) of
Theorem~\ref{Tpmpstrata}. Hence $\pi_{\min}( A )$ is finer than
$\pi^G( A )$, by property (c). Conversely, by
Theorem~\ref{Tpmpstrata}(1), each block of $A$ given by $\pi^G( A )$
lies in a single $G$-orbit, so $\pi^G( A )$ is finer
than~$\pi_{\min}( A )$. Thus $\pi_{\min}( A ) = \pi^G( A )$, as
required.
\end{proof}

The preceding result allows us to extend the Schubert cell-type
stratification to all~$N \times N$ Hermitian matrices.

\begin{definition}
Given $N \geq 1$, $G \subset \C^\times$, and a Hermitian matrix
$A \in \C^{N \times N}$, define the partition $\pi^G( A ) \in \Pi_N$
to be $\pi_{\min}( A )$, as in Proposition~\ref{Pexists}.
Furthermore, given $\pi \in \Pi_N$, let
\[
\stratumsymb^G_\pi := 
\{ \text{Hermitian } A \in \C^{N \times N} : \pi^G( A ) = \pi \}.
\]
\end{definition}

Once again, we have that
$\overline{\stratumsymb^G_\pi} =
\bigsqcup_{\pi' \prec \pi} \stratumsymb^G_{\pi'}$ for all $\pi \in \Pi_N$
and $G \subset \C^\times$, as above.

\begin{remark}
Let $N \ge 3$. The matrix
\[
A = \begin{pmatrix} 1 & 2 & 0 \\ 2 & 8 & 0 \\ 0 & 0 & \Id_{N-2}
\end{pmatrix}
\]
is positive semidefinite, so $3$-PMP. If
$G = \langle 2 \rangle = \{ 2^n : n \in \Z \}$, then
$\pi_{\min}( A ) = \{ \{ 1, 2 \}, \{ 3 \} \}$ and the diagonal block
$A_{\{ 1, 2 \}}$ has rank~$2$. Thus $\pi_{\min}( A ) \neq \pi^G( A )$
and the final part of Proposition~\ref{Pexists} has no immediate
extension beyond the case where $G \subset S^1$.
\end{remark}

\section{The simultaneous kernels of Hadamard powers of $3$-PMP
matrices}\label{Ssimul}

Having understood the Schubert cell-type stratification of the class of
$3$-PMP matrices, our aim below is to compute the simultaneous kernel of
the entrywise powers of any $3$-PMP matrix. This extends
\cite[Theorem~5.7]{BGKP-fixeddim}, which was obtained for positive
semidefinite matrices only. 

\begin{theorem}\label{T3pmp}
Let the Hermitian matrix $A \in \C^{N \times N}$
be $3$-PMP, and let $\pi' = \{ I'_1, \dots, I'_{m'} \}$ be any
partition refined by
$\pi = \pi^{\{ 1 \}}( A ) = \{ I_1, \dots, I_m \}$. The following
spaces are equal.
\begin{enumerate}
\item[\tu{(1)}] The simultaneous kernel of $\one{N}$, $A$, \dots,
$A^{\circ ( N - 1 )}$.

\item[\tu{(2)}] The simultaneous kernel of $A^{\circ n}$ for all
$n \geq 0$.

\item[\tu{(3)}] The simultaneous kernel of the block-diagonal matrices
\smash[b]{$\diag A_{\pi'}^{\circ n} :=
\oplus_{j = 1}^{m'} A^{\circ n}_{I'_j \times I'_j}$}
for all $n \geq 0$.

\item[\tu{(4)}] The kernel of
$J_\pi := \oplus_{j = 1}^m \one{I_j}$.
\end{enumerate}
However, this equality of kernels need not hold for matrices that are
not $3$-PMP.
\end{theorem}

\begin{proof}
That (1) and (2) describe the same subspace follows from
\cite[Lemma~3.5]{BGKP-fixeddim}, which gives
that~$A^{\circ M} = \sum_{j = 0}^{N - 1} D_{M, j}( A ) \, A^{\circ j}$
for certain matrices $D_{M, j}( A )$ and any $M \geq 0$.
We now show equality of the subspaces~(2) and~(4); the same argument,
\emph{mutatis mutandis}, implies that~(2) and~(3) are equal as well.

Note first that $\bv = ( v_1, \dots, v_N )^T \in \ker J_\pi$
if and only if $\sum_{i \in I_j} v_i = 0$ for all
$j \in \{ 1, \dots, m \}$. Since
$A^{\circ n}$ is constant on each block of the form $I_i \times I_j$,
it follows that $\ker J_\pi \subseteq \ker A^{\circ n}$ for
all~$n \geq 0$.

To obtain the reverse inclusion, let $B \in \C^{m \times m}$ denote
the \emph{compression} of $A$,
\begin{equation}\label{Ecompress}
B := \Sigma^\downarrow_\pi(A),
\end{equation}
so that $b_{i j}$ equals the constant value taken in the block
$A_{I_i \times I_j}$. Let $\bv \in \cap_{n \geq 0} \ker A^{\circ n}$ and
write~$\bw := (w_1, \dots, w_m)^T$, with
$w_j := \sum_{i \in I_j} v_i$ for all $j$. Then $\bv \in \ker J_\pi$
if and only if $\bw = \bZ$.
Since $A^{\circ n} \bv = \bZ$ for all $n$, it follows that
$B^{\circ n} \bw = \bZ$ for all $n \geq 0$. Thus the
result follows if~$\cap_{n \ge 0} \ker B^{\circ n} = \{ 0 \}$.

If $( p, q ) \in I_i \times I_j$ then $b_{i j} = a_{p q}$; in
particular, $b_{i i} \geq 0$, and we may assume without loss of
generality that the diagonal entries of $B$ are in non-increasing
order. Supposing for contradiction that $b_{i i} = b_{i j}$, where
$j > i$, the $3$-PMP property for~$A$ gives that
\[
b_{i i}^2 \geq b_{i i} b_{j j} = a_{p p} a_{q q} \geq %
| a_{p q} |^2 = | b_{i j} |^2 = b_{i j}^2 = b_{i i}^2,
\]
so $b_{i i} = b_{i j} = b_{j i} = b_{j j}$. This shows that $B$ is
constant on the block $\{ i, j \} \times \{ i, j \}$, which contradicts
the choice of~$\pi$, by Theorem~\ref{Tpmpstrata}. Hence we have that
\[
b_{i i} \neq b_{i j} \qquad \text{whenever } 1 \leq i < j \leq m.
\]
We claim this implies that
\[
\bigcap_{n = 0}^{m - 1} \ker B^{\circ n} = \{ \bZ \},
\]
which will establish the desired conclusion.

We proceed by induction on $m$, with the base case $m = 1$ being
clear, since $B^{\circ 0} = (1)$. Suppose~$\bu = ( u_j ) \in \C^m$ is
such that $B^{\circ n} \bu = \bZ$ whenever $0 \leq n \leq m-1$. If we
can show that~$u_1 = 0$, then the result follows by a standard
induction argument.

It suffices to assume that $\bu$ is annihilated by $\br^{\circ 0}$,
\dots, $\br^{\circ (m-1)}$, where $\br := ( b_{11}, \dots, b_{1m} )$
is the first row of $B$. Let $\bs := ( s_1, \dots, s_k )$ be a
\emph{compression} of $\br$, which contains each distinct entry in
$\br$ exactly once, and has $s_1 = b_{11}$. Suppose
$\{ 1, \dots, m \} = I''_1 \sqcup \cdots \sqcup I''_k$ is the
corresponding partition, so that
\[
s_j = b_{1i} \quad \iff \quad i \in I''_j \qquad ( j = 1, \dots, k ),
\]
and note that $I''_1 = \{ 1 \}$. If $v_j := \sum_{i \in I'_j} u_i$,
then the vector $\bv = ( v_1, \dots, v_k )^T$ is annihilated
by~$\bs^{\circ 0}$, \dots, $\bs^{\circ (m - 1)}$. The first $k$ of
these are linearly independent, as they form a Vandermonde matrix, so
$\bv = \bZ$. Hence $u_1 = v_1 = 0$, as required.

Finally, to see that the inclusion between the kernels in (4) and (2)
is not always an equality, consider the Toeplitz tridiagonal matrix
$T_N$, with $( i , j )$ entry equal to~$1$ if $| i - j | \leq 1$ and
equal to~$0$ otherwise, for $N \ge 3$. This matrix is readily seen to
be $2$-PMP but not $3$-PMP. Note
that~$\pi = \pi^{\{ 1 \}}( T_N ) = \{ \{ 1 \}, \dots, \{ N \} \}$,
so $\ker J_\pi = \{ 0 \}$. However, if $N = 3 k + 2$ for
some~$k \geq 1$, then
$( 1, -1, 0, 1, -1, 0, \dots, 1, -1 )^T$ lies in~$\ker \one{N}$ and
$\ker T_N^{\circ n}$ for all~$n \geq 1$.
\end{proof}

\begin{remark}
A more general objective is the computation of the simultaneous
kernel of the set of Hadamard powers $\{ A^{\circ n} : n \geq 0 \}$,
where $A \in \F^{M \times N}$ is a linear map between
finite-dimensional vector spaces over a 
field~$\F$. Clearly this reduces to finding the simultaneous kernels of
the Hadamard powers of each row~$\bu^T$ of~$A$. There exists a unique
coarsest partition $\{ I_1, \dots, I_m \} \in \Pi_N$ such that the
entries in~$\bu$ are equal within each block~$I_j$, for all
rows~$\bu^T$ of~$A$. Working as in the proof of Theorem~\ref{T3pmp}
shows that
\[
\bigcap_{\bu^T} \bigcap_{n \geq 0} \ker ( \bu^T )^{\circ n} =
\bigoplus_{j = 1}^m \ker \onesymb_{1 \times I_j}.
\]
This provides a recipe to compute the simultaneous kernel of the
Hadamard powers of an arbitrary rectangular matrix over any field.
The Toeplitz counterexample~$T_N$ in the proof of Theorem~\ref{T3pmp}
shows that, in order to give a more precise description of the
simultaneous kernel, additional assumptions are required, such
as being $3$-PMP.
\end{remark}

In light of this remark, it is noteworthy that the proof of
Theorem~\ref{T3pmp} gives the following result.

\begin{proposition}
Let $\F$ be an arbitrary field, and suppose $A \in \F^{N \times N}$ is
such that
\[
a_{i i} \neq a_{i j} \qquad \text{whenever } 1 \le i < j \le N.
\]
Then $\cap_{n = 0}^{N - 1} \ker A^{\circ n} = \{ \bZ \}$.
\end{proposition}

\begin{remark}
The operator $\Sigma^\downarrow_\pi( \cdot )$ in
Equation~(\ref{Ecompress}) has a number of interesting
properties. These features and their ramifications will be explored
in forthcoming work~\cite{BGKP-strata}.
\end{remark}

\subsection*{Acknowledgments}

The authors thank the International Centre for Mathematical Sciences,
Edinburgh, and Iowa State University, host organization of the 2017 ILAS
conference, for their hospitality, stimulating atmosphere, and excellent
working conditions. D.G. is partially supported by a University of
Delaware Research Foundation grant (UDRF GUILLOT 16-18), by a Simons
Foundation collaboration grant for mathematicians ($\#$526851), and by a
University of Delaware Research Foundation Strategic Initiative grant
(UDRF-SI GUILLOT 17-19). A.K. is partially supported by a Young
Investigator Award from the Infosys Foundation.

\end{document}